%% file: majorant_arxiv_v2.tex
\documentclass[oneside,11pt]{amsart}
\usepackage{amssymb,latexsym,amsmath,amsthm,enumitem,mathrsfs}
\usepackage[text={360pt,615pt},centering]{geometry}
\usepackage{hyperref}
\usepackage{fancyhdr}
\usepackage{color}
\usepackage{stmaryrd}
\usepackage{mathrsfs}
\usepackage{lineno}
\usepackage{enumitem}
\usepackage{soul}
\usepackage[normalem]{ulem}

\input{format}

\definecolor{pink}{rgb}{1,.2,.6}
\definecolor{orange}{rgb}{0.7,0.3,0}
\definecolor{blue}{rgb}{.2,.6,.75}
\definecolor{green}{rgb}{.4,.7,.4}
\definecolor{purple}{RGB}{127,0,255}

\begin{document}

\title[On the strict majorant property]{On the strict majorant property\\ in arbitrary dimensions}

\author[Gressman]{P. Gressman}
\address{University of Pennsylvania, 209 South 33rd Street, Philadelphia PA 19104}
\email{gressman@math.upenn.edu}

\author[Guo]{S. Guo}
\address{University of Wisconsin Madison, Madison, WI}
\email{shaomingguo@math.wisc.edu}

\author[Pierce]{L.B. Pierce}
\address{Duke University, 120 Science Drive, Durham NC 27708 }
\email{pierce@math.duke.edu}

\author[Roos]{J. Roos}
\address{University of Massachusetts Lowell, Lowell, MA 01854 
\&  The University of Edinburgh, Edinburgh EH9 3FD
}
\email{jroos.math@gmail.com}

\author[Yung]{P.-L. Yung}
\address{ Australian National University, Canberra, ACT 2601 
\& The Chinese University of Hong Kong, Shatin, Hong Kong}
\email{polam.yung@anu.edu.au, \, plyung@math.cuhk.edu.hk}

	\begin{abstract}
In this work we study $d$-dimensional majorant properties. We prove that a set of frequencies in $\Z^d$ satisfies the strict majorant property on $L^p([0,1]^d)$ for all $p> 0$ if and only if the set is affinely independent. We further construct three types of violations of the strict majorant property. Any set of at least $d+2$ frequencies in $\Z^d$ violates the strict majorant property on $L^p([0,1]^d)$ for an open interval of $p \not\in 2\N$ of length 2. Any infinite set of frequencies in $\Z^d$ violates the strict majorant property on $L^p([0,1]^d)$ for an infinite sequence of open intervals of $p \not\in 2\N$ of length $2$. Finally, given any $p>0$ with $p \not\in 2\N$, we exhibit a set of $d+2$ frequencies on the moment curve in $\R^d$ that violate the strict majorant property on $L^p([0,1]^d).$
	\end{abstract}
	
	\maketitle

 	\section{Introduction}
This paper introduces the systematic study of  majorant properties on $L^p([0,1]^d)$ in arbitrary dimensions $d$,  motivated by a well-known circle of ideas that is nearly 100 years old. In 1935, Hardy and Littlewood \cite{HarLit35} wrote a brief paper on one-dimensional  majorant inequalities of the form 
\beq\label{HLineq}
	\Big\| \sum_{n \in \Ga} a_n e(n \cdot x)  \Big\|_{L^p([0,1])} \leq \Big\| \sum_{n \in \Ga } A_n e(n \cdot x)  \Big\|_{L^p([0,1])}  
	\eeq
where $\Ga \subset \Z$ is a finite set of frequencies. Here as usual, $e(\theta) := e^{2\pi i \theta}$ for $\theta \in \R$.   Given a set of frequencies $\Ga \subset \Z$ and an exponent $p$, if this inequality holds for all choices of coefficients $a_n, A_n$ with $|a_n| \leq A_n$ for each $n \in \Ga$, then we say the \emph{strict majorant property} holds for $\Ga,p$.
For any finite set $\Ga \subset \Z$, the strict majorant property holds for all $p \in 2\N$ by   a simple expansion of the integral, as Hardy and Littlewood point out.  

Does it also hold for all  $p \not\in 2\N$?   
Hardy and Littlewood write: ``This is untrue and, since it is the falsity of (\ref{HLineq}) which first reveals the difficulties of our problem, we prove it at once...'' for $p=3$. The falsity was verified for all $p\ge 1, p \not\in 2\N$ by Boas \cite{MR149175}, where for the case $1\le p<2$ he referred to Zygmund \cite[page 128, Vol. II]{MR0107776}. 
Hardy and Littlewood suggested instead the study of the \emph{majorant property}, the property that there is some constant $C_p$ such that  (\ref{HLineq}) holds for all $\Ga \subset \Z$ if the right-hand side is enlarged by $C_p$. Landmark work of Bachelis \cite{MR320636}, Mockenhaupt and Schlag \cite{MR2488338}, and Green and Ruzsa \cite{MR2103913} dramatically confirmed that the majorant property is violated for every $p>2, p \not\in 2\N$.
Majorant properties and possible violations of these properties continue to inspire interest, also because of their close relationship to the local restriction conjecture for the sphere and the Kakeya conjecture; see \S \ref{sec_lit}.

\subsection{Main results}

In this paper we  study   strict majorant properties in arbitrarily high dimensions.
Let $\Gamma\subset \Z^d$ be a fixed set  of   $d$-tuples of integers. We say that $\Gamma$ satisfies the {\em strict majorant property}  on $L^p([0,1]^d)$  if for all choices of real coefficients $(a_n)_{n\in\Gamma}$, $(A_n)_{n\in\Gamma}$ with $|a_n| \leq A_n$,
\beq\label{smp}
\Big\| \sum_{n\in\Gamma} a_n e(n\cdot x) \Big\|_{L^p([0,1]^d)}\le \Big\| \sum_{n\in\Gamma} A_n e(n\cdot x) \Big\|_{L^p([0,1]^d)}. 
\eeq
For any set $\Ga \subset \Z^d$, this statement is  true for all $p \in 2\N$ (see \S \ref{sec_remarks}).
The main question is: when $p \not\in 2\N$, for which $\Ga$ is it true?

Our first main result characterizes the sets  $\Ga \subset \Z^d$  for which the strict majorant property holds for all $p>0$.   We recall that  a set $\Ga \subset \Z^d$ is affinely independent if  for any $n_0 \in \Ga$, $\{n-n_0 \in \Z^d \colon n \in \Ga, n \ne n_0\}$ is linearly independent.

\begin{thm}\label{thm_d2}
Fix an integer  $d \geq 1$. 
A non-empty set $\Ga \subset \Z^d$ satisfies the strict majorant property on $L^p([0,1]^d)$ for all $p>0$ if and only if $\Ga$ is affinely independent.
Furthermore, whenever $\Gamma$ is not affinely independent, then there exists an integer $m \ge 0$, and real coefficients $(a_n)_{n \in \Gamma}$, such that for every $p \in (2m,2m+2)$, 
\beq \label{eq:thm1.1}
\left\| \sum_{n \in \Gamma} |a_n| e(n \cdot x) \right\|_{L^p([0,1]^d)} < \left\| \sum_{n \in \Gamma} a_n e(n \cdot x) \right\|_{L^p([0,1]^d)}. 
\eeq
In particular, this holds for every set $\Gamma \subset \Z^d$ of cardinality at least $d + 2$.
\end{thm}
 
If $\Ga \subset \Z^d$ is an infinite set, we  construct counterexamples to the strict majorant property for arbitrarily large $p$ lying in open intervals of length 2.  
 \begin{thm}\label{thm_infinite}
 Fix an integer  $d \geq 1$. If $\Gamma \subset \Z^d$ is infinite, then for infinitely many positive integers $m$, there exist real coefficients $(a_n)_{n\in\Gamma}$ such that for every $p\in (2m, 2m+2)$,
\[
\left\| \sum_{n \in \Gamma} |a_n| e(n \cdot x) \right\|_{L^p([0,1]^d)} < \left\| \sum_{n \in \Gamma} a_n e(n \cdot x) \right\|_{L^p([0,1]^d)}. 
\]
 \end{thm}
 The length of these intervals of $p$ is tight, since 
 the strict majorant property holds for all $p \in 2\N$.

 Third, we prove violations of the strict majorant property for a nice geometric example: the moment curve.  This  relates to recent work of Bennett and Bez \cite{MR3050801}, who  introduced the study of the strict majorant property  for frequencies on the parabola, motivated by connections to the Schr\"odinger equation on the torus and discrete restriction, and also relations to \cite{BBC09}. They proved  that for $\Ga \subset \{ (n,n^2): n \in \Z\}  \subset \Z^2$, for every $p>2$, $p \not\in 2\N$, the strict majorant property fails. 
We generalize this to any dimension: for every  $p \not\in 2\N$, we  exhibit   $d+1$   integral points $n_0, \ldots, n_d$ on the moment curve  in $\R^d$ such that $\Ga = \{{\bf 0}, n_0, \ldots, n_d\}$ fails the strict  majorant property.
 
\begin{thm}\label{thm_moment_curve}
Fix an integer $d \geq 1$. Let $\gamma( t ) = ( t , t^2 ,\dots, t^d)$ parametrize the moment curve in $\R^d$. For any $p > 0$ with $p \notin 2\N$, there exists $k \in \N$ and $a_0, \dots, a_d \in \R$ such that
\[
\Big\| 1 + \sum_{i=0}^d |a_i| e(\gamma(k+i) \cdot x) \Big\|_{L^p([0,1]^d)} < \Big\| 1 + \sum_{i=0}^d a_i e(\gamma(k+i) \cdot x) \Big\|_{L^p([0,1]^d)}.
\]
\end{thm}

Nevertheless, we observe in \S \ref{sec_remarks} that a weaker majorant property does hold: for all choices of real coefficients $a_n,A_n$ with $|a_n| \leq A_n$ for all $n$,
\beq\label{weak_constant} \Big\| \sum_{n \in \Z} a_ne(\ga(n) \cdot x) \Big\|_{L^p([0,1]^d)} \leq (d!)^{1/2d}\Big\| \sum_{n \in \Z} A_ne(\ga(n) \cdot x) \Big\|_{L^p([0,1]^d)} \eeq 
for \emph{all} $2 \leq p \leq 2d$.
 This generalizes the observation in \cite[Thm. 1.2]{MR3050801}.

\subsection{Relation to open problems}\label{sec_lit}

The connection of (1-dimensional) majorant inequalities with the local restriction conjecture for the sphere and  the Kakeya conjecture  arises via a quantitative study of how big a correction factor $B_p(\Ga)$ is needed to make
\beq\label{MS_obs}
\Big\| \sum_{n \in \Ga} a_n e(n  x)  \Big\|_{L^p([0,1])} \leq B_p(\Ga) \Big\| \sum_{n \in \Ga}  e(n  x)  \Big\|_{L^p([0,1])} \eeq hold for \emph{all} choices of 
 coefficients $|a_i| \leq 1$. 
 Mockenhaupt and Schlag \cite[Thm. 3.2]{MR2488338}, and independently Green and Ruzsa \cite{MR2103913} proved that   for every $p>2, p \not\in 2\N,$ for every sufficiently large $N$, there exists   $\al_p>0$,  and a choice of frequency set $\Ga \subset \{1,2,\ldots,N\}$ that requires $B_p(\Ga) \gg N^{\al_p}$.
  Yet if it can be shown that $B_p(\Ga) \ll_{p,\ep}N^\ep$ for all $\ep>0$ for a \emph{particular} $\Ga \subset \{1,2,\ldots,N\}$ relevant to the local restriction conjecture, this will imply the local restriction conjecture and hence the Kakeya conjecture (see \cite{MR2103913},\cite{Mockenhaupt}). 
 Indeed \cite[Thm. 4.4]{MR2488338} does show that random subsets of $\{1,2,\ldots, N\}$ do have this property with a high probability. It is naturally of interest to exhibit specific sets $\Gamma \subset\{1,2,\ldots,N\}$ for which $B_p(\Gamma) \ll_{p,\ep} N^\ep$ (or even smaller). 
 
 In this direction,  for any set of frequencies $\Gamma \subset \{1,2, \ldots, N\}$, for all $p \geq 2$,
(\ref{MS_obs}) holds for all coefficients $|a_i| \leq 1$ with 
\beq\label{MS}
B_p(\Gamma) \ll (N /|\Ga|)^{1/p}.
\eeq
 This is noted in the first display equation of \cite[\S 2, p. 1191]{MR2488338}, and follows from comparing  a consequence of the Hausdorff-Young inequality, 
\[ \Big\| \sum_{n \in \Ga} a_n e(nx)\Big\|_{L^p([0,1])} \leq \Big\|\{a_n\}\Big\|_{\ell^{p'}} \leq |\Ga|^{1/p'} = |\Ga|^{1-1/p}, \]
to the lower bound
\[ \Big\| \sum_{n \in \Ga}  e(nx)\Big\|_{L^p([0,1])} \geq \left( \int_{|x| \ll 1/N} \Big(\sum_{n\in \Ga} \Re(e(nx))\Big)^p dx \right)^{1/p} \gg_p |\Ga|N^{-1/p}.
\]
In particular, if   $|\Ga| \gg N$ (for example an arithmetic progression), then $B_p(\Ga)\ll 1$. Thus the remaining interesting cases to investigate $B_p(\Ga)$ have $|\Ga| = O(N^\rho)$ for some $\rho<1$. In such cases, many immediate corollaries follow from existing results of an arithmetic flavor.
For example, if $\Ga = \mathbb{P}_N$,  the set of prime numbers in the interval $[1,N]$, one sees that $B_p(\mathbb{P}_N) \ll (\log N)^{1/p}$ for all $p \geq 2$. Or, if $Q$ is a fixed (positive definite) binary quadratic form with (fundamental) discriminant $-D<0$, and $\Ga$ denotes the integers in $[1,N]$ represented by $Q$, then $|\Ga| \gg_D N/\sqrt{\log N}$ (for all $N$ sufficiently large relative to $D$), so that $B_p(\Ga) \ll (\log N)^{1/(2p)}$. (See Landau \cite[Vol. 2 p. 643]{Lan09} for $-D=-4$, and more generally Bernays  \cite{Ber12}.) Or, if $\Ga$ denotes the  integers in $[1,N]$ that can be written as a sum of two powerful numbers, then $B_p(\Ga) \ll (\log N)^{\frac{1}{p}(1-2^{-1/3} + o(1))}$ by \cite{Blo05}; see also \cite{BloGra06}. (A number $m$ is powerful if for each $p|m$, $p^2|m$.)

For a given set $\Ga$, it is then interesting to beat (\ref{MS}). 
In the case that $\Ga = \mathbb{P}_N$,  Green \cite[Theorem 1.5]{MR2180408} improved on this, showing that $B_p(\mathbb{P}_N)\ll 1$.
See also earlier work of Bourgain \cite{MR1029904}.
More recent work of Krause-Mirek-Trojan \cite{MR3494246} exhibits a class of deterministic sets $\Ga \subset [1,N]$ with vanishing Banach density as $N \maps \infty$, for which $B_p(\Ga) \ll_p 1$. This concludes our brief remarks on the well-known 1-dimensional setting. 

In the $d$-dimensional setting we study here,  how big a correction factor $B_p(\Ga)$ is required to make a weaker majorant property hold on $L^p([0,1]^d)$ for a particular set of frequencies $\Ga \subset \Z^d$? For the moment curve, (\ref{weak_constant}) shows that $B_p(\Gamma)\ll_d 1$ suffices   for all $2 \leq p \leq 2d.$ For which sets $\Ga \subset [1,N]^d$ does $B_p(\Ga) \ll N^\ep$ suffice? 
This would have interesting applications in recent work of Demeter and Langowski \cite{DemLan21x} on restriction of exponential sums to hypersurfaces (see \cite[Conj. 1.2, Conj. 1.3, Lem. 2.1]{DemLan21x}).

 \subsection{Notation}
We recall that for a positive real number $p \not\in 2\N$,   the generalized binomial coefficient is defined by
\[ \binom{p/2}{0}=1, \qquad \binom{p/2}{j} = \frac{1}{2^j j!} \prod_{\ell=0}^{j-1}(p-2\ell), \qquad j=1, 2, 3,\ldots.\]
  In particular, $\binom{p/2}{j} $ is positive for $0 \leq j \leq \lceil p/2 \rceil$, negative for $j=\lceil p/2 \rceil + 1$, and then alternates in sign for subsequent values of $j$. Second,  multinomial coefficients are defined for $n \in \N$ and $\be = (\be_0,\be_1,\ldots,\be_d) \in \Z^{d+1}_{\geq 0}$ by 
  \[ \binom{n}{\be} = \binom{n}{\be_0,\be_1,\ldots, \be_d} = \frac{n!}{\be_0! \be_1! \cdots \be_d!}.\]
If $n=0$ or $\be=(0,\ldots,0)$, the  multinomial coefficient is  1.  For $b=(b_0, \dots, b_d)$ and $\be \in \Z_{\geq 0}^{d+1}$, then $b^{\beta  }= b_0^{\beta_0  } \dots b_d^{\beta_d  } $   and $|\be| = \be_0 + \cdots + \be_d$. 

If $\Ga$ denotes a (finite or infinite) set in $\Z^d$ and $n \in \Z^d$ then $\Ga + n$ denotes the set $\{ \ga + n : \ga \in \Ga \} \subseteq \Z^d.$ Similarly, if $A \in \Z^{m \times d}$ is an integral $m \times d$ matrix, then $A \Ga$ denotes the set $\{ A\ga : \ga \in \Ga\} \subseteq \Z^m.$ For $n \in \Z^d$, we will use $\widetilde{n}$ to denote the vector $(1,n) \in \Z^{d+1}$. 

\section{Preliminaries}
In this section we prove Proposition \ref{prop_affine_inv}, which is key to proving Theorems \ref{thm_d2} and  \ref{thm_infinite}.
We first collect certain facts we will need about lattices in $\Z^d$ (see e.g. \cite[Ch. 12]{MR1129886} or \cite[Ch. III \S 7]{MR1878556}).
A lattice in $\Z^d$ is a finitely generated, additive subgroup of $\Z^d$, which we will view as a $\Z$-submodule of $\Z^d$. 
Given a finitely generated $\Z$-module $M$, a subset $S$ of $M$ is said to be linearly independent if for every finite set $S' \subset S$ and every choice of coefficients $\{r_s\}_{s \in S'} \subset \Z$, we have $\sum_{s \in S'} r_s s \ne 0$ unless $r_s = 0$ for every $s \in S'$. A subset $S$ of $M$ generates $M$ if every element of $M$ can be written as $\sum_{s \in S'} r_s s$ for some finite set $S' \subset S$ and some  choice of coefficients $\{r_s\}_{s \in S'} \subset \Z$. A basis of $M$ is a linearly independent subset of $M$ that generates $M$. The module $M$ is free if it is isomorphic to $\Z^{d'}$ for some $d'$; in particular  $M$ is free iff it admits a basis. Every basis of a free module has the same cardinality, called the rank of the module.

Finally, every submodule of a free $\Z$-module is free (because $\Z$ is a principal ideal domain), so in particular every lattice in $\Z^d$ is free and has a well-defined finite rank $\le d$.

 We will reduce Theorems \ref{thm_d2} and \ref{thm_infinite}   to  cases when $\Ga$ has full affine dimension (Theorem \ref{thm_affine_indep}) and $\Ga$ is affinely abundant (Theorem \ref{thm_affine_abun}).
 
\subsection{Affine independence}

For $n \in \Z^d$, recall that $\widetilde{n}$ denotes the vector $(1,n) \in \Z^{d+1}$. 

A set $S \subset \Z^d$ is affinely independent if $\widetilde{S} := \{\widetilde{n} \in \Z^{d+1} \colon n \in S\}$ is a linearly independent set in $\Z^{d+1}$; equivalently, for any $n_0 \in S$, $\{n-n_0 \in \Z^d \colon n \in S, n \ne n_0\}$ is linearly independent.
As an example, a set of $d+1$ vectors $\{n_0, n_1, \ldots, n_d\} \subset \Z^d$ is affinely independent if and only if
\beq\label{affine_indep_cond}
\det(\widetilde{n_0}, \widetilde{n_1},  \ldots, \widetilde{n_d})  = \det (n_1 - n_0, \ldots,n_d -n_0)\ne 0.
\eeq
(We see this by subtracting the first column from each of the other columns in the matrix on the left-hand side, and then expanding the determinant along the top row, which has only one nonzero entry.)

The affine dimension of a non-empty set $\Gamma \subset \Z^d$ is the rank of the lattice generated by $\Gamma - n_0$ for any $n_0 \in \Gamma$.  It coincides with the cardinality $|S|-1$, for any maximal affinely independent subset $S$ of $\Gamma$. In particular, $\Ga$ is affinely independent, iff its cardinality is 1 more than its affine dimension.

\begin{prop} \label{prop_affine_inv}
Let $\Gamma \subset \Z^d$ be a non-empty set with affine dimension $d'$. Then there exist $n_* \in \Gamma$, $A \in \Z^{d \times d'}$ of rank $d'$, and a set $\Gamma' \subset \Z^{d'}$ of affine dimension $d'$ with the same cardinality as $\Ga$ such that $\Gamma = n_* + A \Gamma'$.
\end{prop} 

Note that we allow $\Gamma$ to have infinite countable cardinality.

\begin{proof} 
For any $n_* \in \Gamma$, the lattice in $\Z^d$ generated by $\Gamma - n_*$ can be written as $A \Z^{d'}$ for some $A \in \Z^{d \times d'}$ of rank $d'$. The fact that $A$ is injective shows that the affine dimension of the preimage of $\Gamma - n_*$ under $A$ is the same as the affine dimension of $\Gamma$. The set $\Ga'$ is the preimage of $\Gamma - n_*$ under $A$; if $\Ga$ is finite then $|\Ga'| = |\Ga|$; if $\Gamma$ is infinite then so is $\Ga'$.
\end{proof}

\subsection{Affine abundance}\label{sec_affine_abundance}

We say a set $\Gamma \subset \Z^d$ is  affinely abundant, if there exists a $d$-tuple of points $n_1, \dots, n_d \in \Gamma$ such that the set
\[
\{\det \left(\widetilde{n}, \widetilde{n_1}, \dots, \widetilde{n_d} \right) \colon n \in \Gamma\}
\]
is infinite. There is a simple equivalent characterization.

\begin{prop}\label{prop_affine_abundant}
A set $\Gamma\subset\Z^d$ is affinely abundant if and only if it is infinite and has affine dimension $d$. 
\end{prop}

\begin{proof}
It is clear that affinely abundant sets are infinite and have affine dimension $d$. Now suppose $\Gamma\subset\Z^d$ is infinite and has affine dimension $d$.
There exist $m_0,\dots,m_d\in\Gamma$ so that $\det(\widetilde{m_0},\dots,\widetilde{m_d})\not=0$.
By Cramer's Rule:
\[ \widetilde{n} = \sum_{i=0}^d \frac{(-1)^{i} \det (\widetilde n, \widetilde {m_0} ,\ldots,\widehat{\widetilde{m_i}},\ldots,\widetilde{m_d})}{\det (\widetilde{ m_0},\ldots,\widetilde{m_d})} \widetilde{m_i} \qquad \forall n \in \Gamma, \]
where $\widehat{\cdot}$ denotes omission. (By linearity, it suffices to verify this for $\widetilde n = \widetilde m_j$ for some $j=0,\ldots,d$ because all coefficients of $\widetilde{m_i}$ vanish when $j \neq i$ because of a repeated column. The coefficient of $\widetilde{m_j}$ can be seen to equal $1$ by permuting columns of the determinant in the numerator to ``fill the hole.'') 
If for each $i$ the set 
$ \{ \det (\widetilde n, \widetilde {m_0} ,\ldots,\widehat{\widetilde{m_i}},\ldots,\widetilde{m_d}) \, : \, n \in \Gamma \} $
were finite, then every $\widetilde{n}$ would necessarily be one of finitely many distinct linear combinations of $\widetilde{m_0},\ldots,\widetilde{m_d}$, so there could be only finitely many values of $n$ in $\Gamma$.
\end{proof}

\section{A reduction to lower dimensions}
The violation of the strict majorant property indicated in Theorem~\ref{thm_d2} can be reduced to the following result:

\begin{thm}[Full affine dimension case]\label{thm_affine_indep}
Fix an integer  $d \geq 1$. If $\Gamma \subset \Z^d$ has a proper subset whose affine dimension is $d$, then there exists an integer $m \ge 0$, and real coefficients $(a_n)_{n \in \Gamma}$, such that for every $p \in (2m,2m+2)$, 
\[
\left\| \sum_{n \in \Gamma} |a_n| e(n \cdot x) \right\|_{L^p([0,1]^d)} < \left\| \sum_{n \in \Gamma} a_n e(n \cdot x) \right\|_{L^p([0,1]^d)}. 
\]
\end{thm}

Similarly, the proof of Theorem~\ref{thm_infinite} can be reduced to the following:

\begin{thm}[Affinely abundant case]\label{thm_affine_abun}
 Fix an integer  $d \geq 1$. If $\Gamma \subset \Z^d$ is affinely abundant, then for infinitely many positive integers $m$, there exist real coefficients $(a_n)_{n\in\Gamma}$ such that for every $p\in (2m, 2m+2)$,
\[
\left\| \sum_{n \in \Gamma} |a_n| e(n \cdot x) \right\|_{L^p([0,1]^d)} < \left\| \sum_{n \in \Gamma} a_n e(n \cdot x) \right\|_{L^p([0,1]^d)}. 
\]
 \end{thm}

The key to these reductions is the following lemma.

\begin{lemma} \label{lem_affine_invar}
If $\Gamma \subset \Z^d$ and $\Gamma = n_* + A \Gamma'$ for some $n_* \in \Z^d$, $A \in \Z^{d \times d'}$ of rank $d' \leq d$ and a set $\Gamma' \subset \Z^{d'}$, then 
\[
\left\| \sum_{n \in \Gamma} b_n e(n \cdot x) \right\|_{L^p([0,1]^d)} = \left\| \sum_{n' \in \Gamma'} c_{n'} e(n' \cdot x') \right\|_{L^p([0,1]^{d'})}
\]
where $c_{n'} := b_{n_* + A n'}$.
\end{lemma} 

Assuming the lemma for the moment, we  deduce Theorems~\ref{thm_d2} and \ref{thm_infinite}.

\begin{proof}[Proof of Theorem~\ref{thm_d2}]
Suppose $\Gamma \subset \Z^d$ is not affinely independent. Then its cardinality is at least $d' +2$ where $d'$ is the affine dimension of $\Ga$. By Proposition~\ref{prop_affine_inv}, we may find $n_* \in \Gamma$, $A \in \Z^{d \times d'}$ of rank $d'$ and a set $\Gamma' \subset \Z^{d'}$ of affine dimension $d'$ and the same cardinality as $\Gamma$, such that $\Gamma = n_* + A \Gamma'$. Since $\Ga'$ has affine dimension $d'$ and cardinality at least $d'+2$, it contains a proper subset that still has affine dimension $d'$. Now apply the identity of Lemma~\ref{lem_affine_invar} to both sides of the inequality \eqref{eq:thm1.1} claimed in Theorem~\ref{thm_d2}. After these transformations, the resulting inequality is true for all $p \in (2m,2m+2)$ for an appropriate choice of integer $m$ and coefficients $a_n$ by Theorem~\ref{thm_affine_indep}, applied to $\Gamma' \subset \Z^{d'}$ and affine dimension $d'$. In particular, the strict majorant property fails for some $p > 0$.

Conversely, in Proposition \ref{prop_characterize}, we show that if $\Ga \subset \Z^d$ is affinely independent, then the strict majorant property holds for all $p>0$. This completes  the characterization in Theorem \ref{thm_d2}.
\end{proof}

\begin{proof}[Proof of Theorem~\ref{thm_infinite}]
In the setting of Theorem \ref{thm_infinite}, $\Gamma \subset \Z^d$ is infinite. Let $d'$ be the affine dimension of $\Gamma$.
By Proposition~\ref{prop_affine_inv}, we may find $n_* \in \Ga$, $A \in \Z^{d \times d'}$ of rank $d'$, and an infinite set $\Gamma' \subset \Z^{d'}$ of affine dimension $d'$ such that $\Gamma=n_* + A \Gamma'$. 
By Proposition~\ref{prop_affine_abundant}, $\Gamma'$ is affinely abundant.

Since $\Gamma=n_* + A \Gamma'$, we may apply Lemma~\ref{lem_affine_invar} to both  sides of the inequality claimed in Theorem \ref{thm_infinite}. After these transformations, the resulting inequality is true for an infinite sequence of integers $m$ and appropriate coefficients $a_n$ by  Theorem~\ref{thm_affine_abun}, applied to the affinely abundant $\Gamma' \subset \Z^{d'}$.  
\end{proof}

We now prove Lemma~\ref{lem_affine_invar}; and then turn to  Theorems~\ref{thm_affine_indep} and \ref{thm_affine_abun}.  

\subsection{Proof of Lemma~\ref{lem_affine_invar}}
Under the hypotheses of Lemma \ref{lem_affine_invar}, for any coefficients $b_n$ define $c_{n'} = b_{n_* + An'}$. Then for all $x \in [0,1]^d,$
 \[ \Big| \sum_{n \in \Ga} b_n e(n \cdot x) \Big|  = \Big| \sum_{n' \in \Ga' } c_{n'} e((n_* + An') \cdot x) \Big|  = \Big| \sum_{n' \in \Ga' } c_{n'} e(An' \cdot x) \Big|.\]
We now complete $A$ to a $d \times d$ matrix $B$ with integer entries and non-zero determinant by appending $d-d'$ suitable columns of coordinate vectors.
Then for each column vector $n'\in\Gamma'$,
$ An' =  B\, \begin{pmatrix}
n'\\ 0\end{pmatrix}
; $
where $0$ stands for a zero in each of the $d'+1, d'+2, \ldots, d$-th places in the column.
 Then 
  \beq\label{A_to_B} \int_{[0,1]^d} \Big|\sum_{n \in \Ga} b_n e(n \cdot x)\Big|^p dx
   =\int_{[0,1]^d} \Big|\sum_{n' \in \Ga'} c_{n'} e( \begin{pmatrix}
n'\\ 0\end{pmatrix} \cdot B^t x) \Big|^p dx.\eeq
Because $B^t$ is invertible and has integer entries and $e(\cdot)$ is $1$-periodic, an integral over $[0,1]^d$ is invariant under a change of variables that eliminates the matrix; we record this as a lemma, whose proof we defer until the end of the section. 

\begin{lemma}\label{prop:linchgvar}\label{lemma_change_var}
Let $B$ be  a $d \times d$ matrix with integer entries and nonzero determinant. If $F$ is a 1-periodic complex-valued function on $\R^d$, that is   $F(x+m) = F(x)$ for all $m \in \Z^d$, then 
\beq\label{B_int}  \int_{[0,1]^d} F(B x) dx = \int_{[0,1]^d} F(x) dx . \eeq
\end{lemma}
We apply the lemma in (\ref{A_to_B}) to conclude that
\[ \begin{split} \int_{[0,1]^d}  \Big| \sum_{n\in\Gamma} b_n e(n\cdot x)\Big|^p dx  & = \int_{[0,1]^d} \Big|\sum_{n' \in \Ga'} c_{n'} e( \begin{pmatrix}
n'\\ 0\end{pmatrix} \cdot x) \Big|^p dx \\ & = \int_{[0,1]^{d'}} \Big|\sum_{n' \in \Ga'} c_{n'} e(n' \cdot x')\Big|^p dx', \end{split} \]
where the last step follows by Fubini's theorem.  Lemma \ref{lem_affine_invar} is proved.

\subsection{Proof of Lemma \ref{lemma_change_var}}
Let us first suppose that in addition to being a matrix with integer entries and nonzero determinant, $B = (b_{ij})$ is  upper triangular.
Then (\ref{B_int}) holds
by a successive change of variables in the $d$ coordinates of $[0,1]^d$. Indeed, to evaluate
\begin{equation}
\int_{[0,1]^d} F(b_{11} x_1 + \dots + b_{1d} x_d, \dots, b_{dd} x_d) dx, \label{coveq1}
\end{equation}
let $y_1 = b_{11} x_1 + \dots + b_{1d} x_d$; as $x_1$ varies over $[0,1]$, $y_1$ varies over an interval of length $|b_{11}| \in \mathbb{N}$. By periodicity of $F$ in the first variable, \eqref{coveq1} becomes
\[
\int_{[0,1]^d} F(y_1, b_{22} x_2 + \dots + b_{2d} x_d, \dots, b_{dd} x_d) dy_1 dx',
\]
in which $x' = (x_2,\ldots, x_d)$.
Now  by successively setting $y_j = b_{jj} x_j + \dots + b_{jd} x_d$ for $j= 2, \dots, d$, this proves the claim. 

In general, we must show that given any matrix $A$ with integer entries and nonzero determinant, there exists a matrix $E \in \GL(d,\Z)$, given by a product of elementary matrices (in particular, of  determinant $\pm1$), such that $A= EB$, where $B$ is upper triangular  (and has integer entries and nonzero determinant). For then certainly 
$
\int_{[0,1]^d} F(A x) dx = \int_{[0,1]^d} F(B x) dx,$
and we may apply the special case proved above.

This is essentially the claim that such a matrix $A$ can be put in  Hermite normal form by applying only elementary row operations corresponding to matrices in $\mathrm{GL}(d,\Z)$, i.e.  we can transform $A$  to an upper triangular matrix $B$ with integer entries and nonzero determinant  in finitely many steps, using only row swaps and replacing a row by its sum with a multiple of another row. We describe this process. First we swap rows until the entry in the first column with the smallest absolute value is brought to the first row. Then we subtract off multiples of the first row from the remaining rows to reduce the absolute values of the first entries of the remaining rows as much as possible. Then we iterate this process (effectively, running the Euclidean algorithm on the entries of the first column). After finitely many steps,  the top entry of the first column is the gcd of all the entries in the first column of the original matrix $A$, and the remaining entries are zero. We may now repeat on the $(d-1) \times (d-1)$ minor obtained from the above matrix by deleting its first row and first column. After finitely many iterations, we obtain an upper triangular matrix $B$ with integer entries and nonzero determinant.

\section{Initial setup to prove the main theorems}\label{sec_setup}

In the proofs of Theorems \ref{thm_affine_indep} and \ref{thm_affine_abun}, fix $a_{n_{\bullet}} = 1$ for some $n_{\bullet} \in \Gamma$,  choose suitable $n_0, \dots, n_d \in \Gamma$ and set $a_n = 0$ for all $n \notin \{n_{\bullet}, n_0, \dots, n_d\}$. By using that $|e(n_{\bullet} \cdot x)| = 1$ and renaming $n_i-n_{\bullet}$ as $n_i$ for $i = 0, \dots, d$, without loss of generality we may assume $n_{\bullet} = 0$. Then it remains to show that
\[
  \Big\| 1+  \sum_{i=0}^d |a_{n_i}| e(n_i\cdot x) \Big\|_{L^p([0,1]^d)}^p <   \Big\| 1+  \sum_{i=0}^d a_{n_i} e(n_i\cdot x) \Big\|_{L^p([0,1]^d)}^p
\]
for some suitable coefficients $a_{n_0}, \dots, a_{n_d}$. The $a_{n_0}, \dots, a_{n_d}$ will be chosen to be small; for this reason, and to prepare ourselves for the subsequent proof of Theorem~\ref{thm_moment_curve}, it helps to understand the following Taylor expansion for
\begin{equation}
\Big\| 1+  \sum_{i=0}^d b_i e(n_i\cdot x) \Big\|_{L^p([0,1]^d)}^p \label{totaylor}
\end{equation}
around $(b_0, \dots, b_d) = (0, \dots, 0)$.

Fix any $p>0$. The expression \eqref{totaylor} equals
\begin{align} 
& \int_{[0,1]^d} \Big( 1 +  \sum_{i=0}^d b_i e(n_i \cdot x) \Big)^{p/2} \Big( 1 +  \sum_{j=0}^d b_j e(-n_j \cdot x) \Big)^{p/2} dx \nonumber \\
&= \int_{[0,1]^d} \sum_{\ell, m \geq 0} \binom{p/2}{\ell} \binom{p/2}{m} \Big( \sum_{i=0}^d b_i e(n_i \cdot x) \Big)^{\ell} \Big( \sum_{j=0}^d b_j e(-n_j \cdot x) \Big)^{m} dx \nonumber \\
&= \sum_{\beta, \gamma \in \mathbb{Z}_{\geq 0}^{d+1}}  \binom{p/2}{|\beta|} \binom{p/2}{|\gamma|} \binom{|\beta|}{\beta} \binom{|\gamma|}{\gamma} b^{\beta +  \gamma} I(\beta-\gamma), \label{Taylor}
\end{align}
in which $b^{\be+\ga} = b_0^{\be_0+\ga_0}\cdots b_d^{\be_d+\ga_d}$, and according to our fixed $n_0,n_1,\ldots, n_d \in \Z^d$ we define $I(u)$ for any  $u = (u_0,\ldots, u_d) \in \R^{d+1}$ by
\[
I(u) := \int_{[0,1]^d} e \Big(\sum_{i=0}^d u_i n_i \cdot x \Big) dx.
\]
The above application of the expansion of the $p/2$ power and the interchange of the infinite sum over $\ell, m$ with the integral can be justified as long as $|b| < 1$ (because this guarantees uniform convergence, over $[0,1]^d$, of the series under the integral on the second line of the display).

Our goal is to isolate a main term in the right-hand side of (\ref{Taylor}), plus a negligible remainder term. Then we aim to show that when we substitute in $a_{n_i}$ for $b_i$ in the main term, we get a strictly larger expression than when we substitute in $|a_{n_i}|$. We will do so by showing that the difference between the expression with $a_{n_i}$ minus the expression with $|a_{n_i}|$ is controlled by a signed product of generalized binomial coefficients (see (\ref{pc})). We will exploit the fact that the generalized binomial coefficients oscillate in sign in order to show that this difference can be forced to be positive in the settings of our theorems. This basic framework agrees with the argument used by Bennett and Bez in the case of the parabola, but  how we isolate the main term, and how we show the difference can be forced to be positive, is novel.

The integral $I(\beta-\gamma)$ is equal to 1 if and only if $\beta,\gamma$ satisfy
 \beq\label{bega}
 \sum_{i=0}^d (\beta_i - \gamma_i) n_i  = \zerobf \in \Z^{d},
 \eeq
and is equal to 0 otherwise.
We now characterize for which $\be,\ga \in \Z_{\geq 0}^{d+1}$ the relation (\ref{bega}) holds. Certainly (\ref{bega}) holds if $\be = \ga$. 
In general we  will show via linear algebraic considerations that $I(\beta - \gamma) = 1$  if and only if $\beta - \gamma$ is an integer multiple of $(c_0, \dots, c_d)$, for a   vector $c \in \Z^{d+1}$ we now construct.

We recall that $n_0,\ldots, n_d \in \Z^d$ (regarded as column vectors) are fixed. Consider the
 the linear functional 
\beq\label{1v_generic}
(x_0, \dots, x_d) \in \Q^{d+1} \mapsto \det\left( 
\begin{array}{cccc}
x_0 & x_1 & \dots & x_d \\
\vert & \vert &  & \vert \\
n_0 & n_1 & \dots & n_d  \\
\vert & \vert &  & \vert \\
\end{array} 
\right) \in \Q,
\eeq 
which for some  fixed $v = (v_0, \dots, v_d) \in \Z^{d+1}$ can be expressed as $(x_0,\ldots, x_d) \cdot (v_0,\ldots, v_d)$ for all $(x_0,\dots,x_d) \in \Z^{d+1}$. 
First, note that $v$ (regarded as a column vector) is in the null-space of the $d \times (d+1)$ matrix $(n_0,\ldots, n_d).$ (Indeed, for each $1 \leq i \leq d,$ the $i$-th coordinate of $(n_0,\ldots, n_d)v$ is the value of the linear functional with $x$ taken to be identical to the $i$-th row of the matrix in (\ref{1v_generic}), so that the determinant is certainly zero.)
Second, note that
\beq\label{affine_cond}
(1,\dots,1) \cdot (v_0, \dots, v_d) = \det(\widetilde{n_0}, \widetilde{n_1},  \ldots, \widetilde{n_d})  = \det (n_1 - n_0, \ldots,n_d -n_0).
\eeq
We now \emph{assume} that the determinant above is nonzero, and proceed under this assumption. Then in particular we see that $v \neq \mathbf{0} \in \Z^{d+1}$ and that 
 the 
\beq\label{rank}
 \text{$d \times (d+1)$ matrix $( n_0, \ldots, n_d)$ has rank $d$.} 
 \eeq
 Consequently the null-space of $( n_0, \ldots, n_d)$ is $1$-dimensional, and hence the vector $v$ constructed above spans the null space.

We have learned that $I(u)=1$ for $u \in \R^{d+1}$ if and only if $u$ is in the null-space of the matrix $(n_0,\ldots, n_d)$, which is if and only if $u$ is a multiple of $v$. We are interested only in  $u \in \Z^d$ and thus it is efficient to find a primitive basis element for the null space, that is, with relatively prime entries. 
Thus we
let $D$ denote $\gcd(v_0, \dots, v_d)$, and finally define 
\beq\label{c_dfn}
c= (c_0, \dots, c_d) := D^{-1} (v_0, \dots, v_d)  \in \Z^{d+1}\setminus{\zerobf} .
\eeq
Always under the assumption the determinant in (\ref{affine_cond}) is nonzero, we conclude that  for $\be, \ga \in \Z_{\geq 0}^{d+1}$, $I(\be - \ga) =1$ if and only if $\be - \ga$ is an integral multiple  of $c$, say $kc$ for $k\in \Z$. 

The case $k=0$ corresponds to $\be = \ga$. The contribution to the Taylor expansion (\ref{Taylor}) from the terms with $\be = \ga$ is 
\[ \sum_{ \beta \in \mathbb{Z}_{\geq 0}^{d+1} }  \binom{p/2}{|\beta|}^2 \binom{|\beta|}{\beta}^2 (b^{\beta})^2  .\]
Because each entry is squared in this contribution, if we substitute in the coefficients $a_i$ or $|a_i|$ for $b_i$, the expression is identical, leading to an identical contribution to the $L^p([0,1]^d)$ norms we are studying. 

Thus to understand when the   majorant property is violated, we aim to isolate the contribution of  lowest-order terms in the expansion (\ref{Taylor}) with $\be \neq \ga$; that is, the contribution from $\be, \ga$ such that $\be - \ga = kc$ for some integer $k \neq 0$, such that $\be + \ga$ has smallest total degree. For any $k$ with $|k | \geq 1$,  if $\be - \ga = kc$, the triangle inequality shows that for each $0 \leq i \leq d$, 
\[ |c_i| \leq |k| |c_i| = |kc_i| = |\be_i - \ga_i| \leq |\be_i| + |\ga_i| = \be_i + \ga_i,\]
since $\be_i, \ga_i \geq 0$. We deduce that \emph{as a polynomial}, $x_0^{|c_0|} x_1^{|c_1|} \cdots x_d^{|c_d|}$ is a divisor of $x^{\be + \ga}$ for every $\be,\ga$ such that $\be - \ga = kc$ for some integer $k \neq 0$. In particular the  lowest-order terms with $\be \neq \ga$ must come from $\be, \ga$ such that $\be + \ga = (|c_0|,  |c_1| ,\ldots, |c_d|)$. Observe that since $c \neq \zerobf \in \Z^{d+1}$, in order to simultaneously satisfy for some integer $k \neq 0$ the three conditions that 
\[ \be + \ga = (|c_0|,  |c_1| ,\ldots, |c_d|), \qquad \be - \ga = k(c_0,\ldots, c_d), \qquad \be_i, \ga_i \geq 0 ,\]
it must be the case that $k = \pm 1$. Then for example, if $k=1$, we must have that $\be_i = c_i$ and $\ga_i=0$ for those $i$ such that $c_i>0$, and $\be_i = 0$ and $\ga_i = -c_i$ for those $i$ such that $c_i< 0$, and finally $\be_i  = \ga_i  =0$ if $c_i=0$. An analogous conclusion is obtained if $k = -1$.

At this point, it is helpful to define the notation
\beq\label{c_notation}
c = c_+ - c_-
\eeq
where $c_+, c_- \in \Z^{d+1}$ are defined by
\[
(c_+)_i := \begin{cases} c_i &\quad \text{if $c_i \geq 0$} \\ 0 & \quad \text{if $c_i < 0$} \end{cases} 
\quad \text{and} \quad
(c_-)_i := \begin{cases} 0 &\quad \text{if $c_i  \geq 0$} \\ -c_i & \quad \text{if $c_i < 0$} \end{cases} .
\]
Then $(|c_0|,  |c_1| ,\ldots, |c_d|) = c_- + c_+$. Note that $c, c_+$ and $c_-$ are used as multi-indices, so that in what follows, $|c|$ denotes the order $|c|=|c_0| + \cdots +|c_d|$, and similarly for the other two.
(At this point we note that at the level of generality of the present discussion, it could be that one of $c_+$ or $c_-$ is the zero vector, but it cannot be that both are the zero vector.)

We apply this 
in  \eqref{Taylor}, distinguishing between the cases $\beta = \gamma$ and those where $\beta \ne \gamma$, and isolating the lowest-order terms characterized above:
\begin{align*}
\Big\|  1 +  \sum_{i=0}^d b_i e(n_i \cdot x) & \Big\|_{L^p([0,1]^d)}^p
= \, \sum_{ \beta \in \mathbb{Z}_{\geq 0}^{d+1} }  \Big[ \binom{p/2}{|\beta|} \binom{|\beta|}{\beta} b^{\beta} \Big]^2 \\
& + 2 \binom{p/2}{|c_-|} \binom{p/2}{|c_+|} \binom{|c_-|}{c_-} \binom{|c_+|}{c_+} b^{c_-+c_+} + o(|b^{c_- + c_+}|).
\end{align*}
Here we have now made the assumption that each coefficient $b_i$ has $|b_i|<1$, so   that the higher order terms, which we noted above are divisible by $b^{c_-+c_+}$ (as a polynomial), are indeed $o(|b^{c_- + c_+}|)$. 

We now apply this with two choices for the coefficients $b_i$: first, real coefficients $a_{n_i}$, and second, the absolute values $|a_{n_i}|$, and then we take the difference. 
We summarize the discussion thus far as a proposition.

\begin{prop}\label{prop_summary}
Suppose that $n_0,\ldots, n_d \in \Z^d$ (regarded as column vectors) have the property that  
\beq \label{affine_cond_sum}
\det(\widetilde{n_0}, \widetilde{n_1},  \ldots, \widetilde{n_d}) \neq 0.
\eeq
Then there exists $c \in \Z^{d+1}\setminus\{0\}$ with $\gcd (c_0,\ldots, c_d)=1$   with corresponding notation $c=c_+ -c_-$ with $c_-,c_+ \in \Z_{\geq 0}^{d+1}$ as in (\ref{c_notation})  such that for any $p \not\in 2\N$, and any real coefficients $a=(a_{n_0},\ldots, a_{n_d}) \in \R^{d+1}$ with  $|a| < 1$, 
\begin{multline}\label{main_id}
 \Big\| 1 + \sum_{i=0}^d a_{n_i} e(n_i \cdot x) \Big\|_{L^p([0,1]^d)}^p - \Big\| 1 + \sum_{i=0}^d |a_{n_i}| e(n_i \cdot x) \Big\|_{L^p([0,1]^d)}^p \\
= \, -2\binom{p/2}{|c_-|} \binom{p/2}{|c_+|} \binom{|c_-|}{c_-} \binom{|c_+|}{c_+} (|a^{c_-+c_+}| - a^{c_-+c_+}) + o(|a^{c_- + c_+}|).
\end{multline}

\end{prop}
 
Now,  in order to prove a set $\{0, n_0, \dots, n_d\} \subset \Z^d$ violates the   majorant property, we     verify that the main term on the right-hand side of (\ref{main_id}) is \emph{positive}, and then take the coefficients $a_i$ sufficiently small so that the $o(|a^{c_- + c_+}|)$ remainder term is dominated by the main term. 
The multinomial coefficients $\binom{|c_-|}{c_-}, \binom{|c_+|}{c_+}$ are positive integers. Moreover, we claim that for $c$ as constructed above, we can (for example) choose   a small $a \in \R^{d+1}$ such that 
\beq\label{ac}
|a^{c_-+c_+}| - a^{c_-+c_+} = 2|a^{c_-+c_+}| > 0.
\eeq
This is because since $c \neq \zerobf \in \Z^{d+1}$ and $\gcd(c_0,\ldots, c_d)=1$,  at least one $c_i$ is odd. 
Thus we simply choose $a$ so that $a_{n_i}<0$ for one index such that $c_i$ is odd, and then choose $a_{n_i}>0$ for all other indices.   
Then to show that the main term on the right-hand side is positive, it suffices to show that 
  for the vector $c$ obtained from the set $\{n_0, \dots, n_d\}$, and the $p \not\in 2\N$ of interest, 
\beq\label{pc}
- \binom{p/2}{|c_-|} \binom{p/2}{|c_+|} > 0,
\eeq
or equivalently,
\beq\label{pcnew}
\begin{cases}
\left(\frac{p}{2} - |c_+|\right) \left(\frac{p}{2} - |c_+| - 1 \right) \cdots \left(\frac{p}{2} - (|c_-| - 1)\right) < 0 \quad \text{if $|c_+| \leq |c_-|$} 
\\
\left(\frac{p}{2} - |c_-|\right) \left(\frac{p}{2} - |c_-| - 1 \right) \cdots \left(\frac{p}{2} - (|c_+| - 1)\right) < 0 \quad \text{if $|c_+| > |c_-|$}.
\end{cases}
\eeq
To prove this, we aim to verify one of the following sufficient conditions:
\begin{enumerate}
\item[(i)] (for Theorems~\ref{thm_affine_indep} and \ref{thm_affine_abun}) one of the two numbers $p/2-(|c_-|-1)$ and $p/2-(|c_+|-1)$ is positive, and the other is between $0$ and $-1$ (so exactly one of the terms in the product in \eqref{pcnew} is negative);
\item[(ii)] (for Theorem \ref{thm_moment_curve}) $|c_-|$ and $|c_+|$ are both bigger than $p/2$, and    they have opposite parities (so that \eqref{pcnew} is a product of an odd number of negative numbers).
\end{enumerate}

\section{Theorem \ref{thm_affine_indep}: the affinely independent case}
For Theorem \ref{thm_affine_indep},  the hypotheses allow us to
choose $n_{\bullet} \in \Ga$ and an affinely independent subset $\{n_0, \ldots, n_d\} \subset \Ga \setminus \{n_{\bullet}\}$ with $d+1$ elements, so that \eqref{affine_cond_sum} holds. 
We need only show that $\{n_{\bullet},n_0,\dots,n_d\}$ violates the strict majorant property on $L^p([0,1]^d)$ for some values of $p$. Set $a_{n_{\bullet}} = 1$. We will choose small real numbers $a_{n_0}, \dots, a_{n_d}$ and some integer $m \geq 0$ so that 
\[
 \Big\| e(n_{\bullet} \cdot x) + \sum_{i=0}^d a_{n_i} e(n_i \cdot x) \Big\|_{L^p([0,1]^d)} > \Big\| e(n_{\bullet} \cdot x) + \sum_{i=0}^d |a_{n_i}| e(n_i \cdot x) \Big\|_{L^p([0,1]^d)}
\] 
for all $p \in (2m,2m+2)$.
Since $|e(n_{\bullet} \cdot x)| = 1$, by considering the translate $n_i - n_{\bullet}$ in place of $n_i$ for every $i = 0, \dots, d$, without loss of generality we may assume $n_{\bullet} = \zerobf$ so that $e(n_{\bullet} \cdot x) = 1$. We then appeal to Proposition~\ref{prop_summary} with the  $n_0,n_1,\ldots, n_d \in \Z^d \setminus \zerobf$ chosen in this way.

For such $n_0,n_1,\ldots, n_d$, we construct  $c \in \Z^{d+1}$ with $\gcd(c_0,\ldots, c_d)=1$ as in (\ref{c_dfn}) above, and focus attention on $a = (a_{n_0}, \dots, a_{n_d})\in \R^{d+1}$ such that (\ref{ac}) holds (which, again, is possible since at least one entry in $c$ is odd). 
In this case (\ref{affine_cond}) shows that
 \[
 c_0 + \dots + c_d = D^{-1} (1, \dots, 1) \cdot (v_0, \dots, v_d) = D^{-1} 
\det(\widetilde{n_0}, \widetilde{n_1}, \dots, \widetilde{n_d}) \ne 0.
 \]
It follows that
$
|c_+| \ne |c_-|,
$
for otherwise $c_0 + \dots + c_d = 0$. Thus $|c_+|$ and $|c_-|$ are distinct non-negative integers. 

To proceed further we claim that
\begin{equation}\label{lower_bd_2}
\max\{|c_-|,|c_+|\} \geq 2.
\end{equation}
Suppose on the contrary that this is not true, so that $\{|c_-|,|c_+|\} = \{0,1\}$.  Then there exists precisely one index $i_0$ with a nonzero entry $c_{i_0} = \pm 1$, and all other coordinates of $c$ are zero. Recall the construction of $c = D^{-1}v$ where $v$ is in the nullspace of the $d \times (d+1)$ matrix $M=(n_0,n_1, \ldots, n_d)$;  in particular, this means $v$ is orthogonal to all the rows $N_1, \ldots, N_d$ of $M$. Consequently, if $c$ has precisely one nonzero entry $c_{i_0} = \pm 1$, this implies that all the rows $N_1, \ldots, N_d \in \R^{d+1}$ are orthogonal to the $i_0$-th coordinate vector in $\R^{d+1}$. In particular, each $N_i$ has a zero in its $i_0$-th coordinate, and  the $i_0$-th column vector in $M$, namely $n_{i_0},$ is the zero vector.   This contradicts our initial hypotheses on $n_0, n_1,\ldots, n_d$.
As a result, we may conclude that $\max\{|c_-|,|c_+|\} \geq 2$, as claimed. 

Now we define
\[
m_+ := \max\{|c_-|,|c_+|\} \quad \text{and} \quad m_- := \min\{|c_-|,|c_+|\};
\] 
hence $m_+ \geq 2$ and $m_+>m_-$.
Choose any $p$ such that 
\beq\label{p_choice}
m_- \leq m_+ - 1 < \frac{p}{2} + 1 < m_+,
\eeq
or equivalently $2m_+-4<p<2m_+-2;$
note $p>0$ since $m_+ \geq 2$. 
Then 
\[
\binom{p/2}{m_+} = \frac{1}{m_+ !} \frac{p}{2} \left(\frac{p}{2} -1 \right) \cdots \left(\frac{p}{2} - (m_+ - 1)\right) < 0,
\]
in which the last factor is negative but all other factors are positive. On the other hand $\frac{p}{2} + 1 > m_-$, so that
\[
\binom{p/2}{m_-} = \frac{1}{m_- !} \frac{p}{2} \left(\frac{p}{2} -1 \right) \cdots \left(\frac{p}{2} - (m_- - 1)\right) > 0.
\] 
This verifies the crucial property (\ref{pc}). 
Finally, we choose $a$ so that (\ref{ac}) holds and $|a|$ is sufficiently small that the error $o(|a^{c_- + c_+}|)$ in Proposition \ref{prop_summary} can be dominated by the main term on the right-hand side. This shows 
\[
 \Big\| 1 + \sum_{i=0}^d a_{n_i} e(n_i \cdot x) \Big\|_{L^p([0,1]^d)} > \Big\| 1 + \sum_{i=0}^d |a_{n_i}| e(n_i \cdot x) \Big\|_{L^p([0,1]^d)},
\] 
and concludes the proof of Theorem \ref{thm_affine_indep}.\\

Note that within this argument, once the set   of affinely independent vectors is fixed, the vector $c$ is fixed, so that $m_+, m_-$ are bounded. Consequently there is an upper bound on the $p$ for which we can verify the crucial property (\ref{pc}), and hence an upper bound on the $p$ for which this method can show the strict majorant property can be violated.

If we instead consider an affinely abundant set of integers, we have the freedom to choose infinitely many elements $n$ in the set we consider, each of which yields a corresponding pair $m_+^{(n)}, m_-^{(n)}$, and such that the corresponding $m_+^{(n)}$ form a strictly increasing sequence; this allows us to construct violations of the strict majorant property for  open intervals of arbitrarily large $p$. 
Similarly,  in the proof of Theorem \ref{thm_moment_curve}, we can take $p$ arbitrarily large by  shifting our focus to a set of points that are sufficiently ``high'' on  the moment curve.

\section{Theorem \ref{thm_affine_abun}: The affinely abundant case}

Recall from \S \ref{sec_affine_abundance} that a set of points $\Ga \subset {\mathbb Z}^d$ is called affinely abundant when there exists a $d$-tuple of points $n_1,\ldots,n_{d} \in \Ga$ such that the set
\beq\label{set_infinite}
\{ \det (\widetilde{n}, \widetilde{n_1}, \ldots, \widetilde{n_d}) \ : \ n \in \Ga \} \eeq
is infinite. Since the values in this set are integers, the values must then become arbitrarily large in absolute value.

Suppose that a set $\Gamma \subset \Z^d$ is affinely abundant, and denote a distinguished $d$-tuple in $\Ga$ with the above property by $\{n_1,\ldots, n_d\}$. 
Fix some $n_{\bullet} \in \Ga \setminus \{n_1, \ldots, n_d\}$. We will show that for infinitely many integers $m \ge 0$, we can choose $n_0 \in \Ga \setminus \{n_{\bullet}, n_1, \ldots, n_d\}$, such that $\{n_{\bullet},n_0,n_1,\dots,n_d\}$ violates the strict majorant property on $L^p([0,1]^d)$ for all $p \in (2m,2m+2)$. As in the previous section, without loss of generality we may  consider from now on $n_i - n_\bullet$ in place of $n_i$, and  $n_{\bullet} = \zerobf$. We will let $a_{n_{\bullet}}=1$. Upon choosing $n_0 \in \Ga \setminus \{\zerobf, n_1, \ldots, n_d\}$, we will choose small real numbers $a_{n_0},\dots,a_{n_d}$ and an appropriate integer $m \ge 0$ depending on $n_0,$ such that 
\[
 \Big\| 1 + \sum_{i=0}^d a_{n_i} e(n_i \cdot x) \Big\|_{L^p([0,1]^d)} > \Big\| 1 + \sum_{i=0}^d |a_{n_i}| e(n_i \cdot x) \Big\|_{L^p([0,1]^d)}
\] 
for all $p \in (2m,2m+2)$.
The key is to choose $n_0$ so that we can apply the analysis of \S \ref{sec_setup}, and finally to observe that the corresponding $m$   can be made arbitrarily large, by varying the choice of $n_0 \in \Ga.$

We will not yet specify which $n_0$ we choose to distinguish, but suppose momentarily that such a choice has been made and fix the set $\{n_0, n_1, \ldots, n_d\}$ of $d+1$ (nonzero) elements in $\Ga \subset \Z^d$. With this set, we follow the construction in \S \ref{sec_setup} and 
define $v=v^{(n_0)} = (v_0,\ldots, v_d) \in \Z^{d+1}$ so that 
\beq\label{dfn_v_affine_abun}
(1,\ldots, 1) \cdot (v_0,\ldots, v_d) = 
\det(\widetilde{n_0}, \widetilde{n_1}, \dots, \widetilde{n_d})
\eeq
as in (\ref{affine_cond}).
 Since $\Gamma$ is affinely abundant, there are infinitely many choices of $n_0 \in \Ga \setminus \{\zerobf \}$ such that this determinant is nonzero, giving $v^{(n_0)} \neq \mathbf{0} \in \Z^{d+1}$, and we assume that $n_0$ has this property. Moreover, since the set (\ref{set_infinite}) is infinite, we can choose a sequence of $n_0$ so that the   construction of $v^{(n_0)}$ yields a sequence of values for $|(1,\ldots,1) \cdot v^{(n_0)}|$ that grows arbitrarily large. 
 
 Now note that within $v^{(n_0)}$, by (\ref{1v_generic}) the coordinate $v_0$ can be computed as the minor of the matrix in (\ref{dfn_v_affine_abun}) that omits the first row and the first column; this minor is independent of $n_0$. In particular, the coordinate $v_0$ stays fixed independent of $n_0$, and thus the $\gcd$ of the coordinates  of $v^{(n_0)}$, namely 
 \[D^{(n_0)} = \gcd(v_0,v_1, \ldots, v_d),\]
 stays \emph{bounded} uniformly in $n_0 \in \Ga \setminus \{ \zerobf\}.$ 
 We now define the corresponding vector $c = c^{(n_0)} = (c_0,\ldots,c_d)$ as in 
(\ref{c_dfn}) by 
\[ c^{(n_0)} = (D^{(n_0)})^{-1} v^{(n_0)} \in \Z^{d+1} \setminus \{\zerobf\},\]
and   the vector $c_+^{(n_0)}$ recording the non-negative coordinates and the vector $c_-^{(n_0)}$ recording the negative coordinates as in (\ref{c_notation}).
Since 
\[c_0 + \dots + c_d = (D^{(n_0)})^{-1}(v_0 + \cdots +v_d) \neq 0\]
we again learn that $|c_+^{(n_0)}|$ and $|c_-^{(n_0)}|$ are distinct non-negative integers. 
Moreover, since $|(1,\ldots,1) \cdot v^{(n_0)}|$ can be made arbitrarily large by choosing $n_0$ from the affinely abundant set $\Ga \setminus \{\zerobf\}$, while $D^{(n_0)}$ is bounded uniformly for all $n_0 \in \Ga \setminus \{\zerobf\}$, then $|c^{(n_0)}|$ can be made arbitrarily large. 
Consequently, if for each $n_0 \in \Ga \setminus \{\zerobf\}$ we define 
\[ m_+^{(n_0)} := \max \{ |c_+^{(n_0)}|, |c_-^{(n_0)}|\} \quad \text{and} \quad m_-^{(n_0)} := \min \{ |c_+^{(n_0)}|, |c_-^{(n_0)}|\},\]
then the set of positive integers
$ \{m_+^{(n_0)} : n_0 \in \Ga \setminus \{\zerobf\}\}$
is infinite. In particular, there is an ordered infinite sequence of choices of $n_0 \in \Ga \setminus \{\zerobf\}$ for which the corresponding sequence of pairs $(m_+^{(n_0)},m_-^{(n_0)})$ has the properties that each $m_+^{(n_0)}\geq 2$,
each pair has $m_+^{(n_0)}>m_-^{(n_0)},$ and moreover the sequence of integers $m_+^{(n_0)}$ is strictly increasing.

Now for each choice of $n_0$, we apply Proposition \ref{prop_summary}  to the $d+1$ vectors $\{n_0,n_1,\ldots, n_d\}$. We choose the coefficients $a_n$ to be zero for all $n \in \Ga \setminus\{\zerobf,n_0,n_1,\ldots,n_d\}$.
After such a choice, we can conclude that
\[ \Big\| \sum_{n \in \Ga} a_n e(n \cdot x)\Big\|_{L^p([0,1]^d)} - 
\Big\| \sum_{n \in \Ga} |a_n| e(n \cdot x)\Big\|_{L^p([0,1]^d)}
\]
is equal to 
\[ \begin{split} -2\binom{p/2}{|c_-^{(n_0)}|} & \binom{p/2}{|c_+^{(n_0)}|} \binom{|c_-^{(n_0)}|}{c_-^{(n_0)}} \binom{|c_+^{(n_0)}|}{c_+^{(n_0)}} (|a^{c_-^{(n_0)}+c_+^{(n_0)}}| - a^{c_-^{(n_0)}+c_+^{(n_0)}}) \\ & + o(|a^{c_-^{(n_0)} + c_+^{(n_0)}}|) \end{split} \]
(where we wrote $a := (a_{n_0},\dots,a_{n_d})$).
We again  focus on choosing coefficients $a \in \R^{d+1}$ such that the factor in the first term that depends on the coefficients $a$ is positive (which, as mentioned earlier in (\ref{ac}), is possible since the gcd of the coordinates of $c^{(n_0)}$ is 1, so at least one entry in $c^{(n_0)}$ is odd). 
Thus the strict majorant property is violated for any $p$ such that 
\[-\binom{p/2}{|c_-^{(n_0)}|} \binom{p/2}{|c_+^{(n_0)}|}>0.\]
We  apply the argument   developed in \eqref{p_choice}: for each $n_0$, any $p$ with
\beq\label{p_choice2}
m_+^{(n_0)} - 1 < \frac{p}{2} + 1 < m_+^{(n_0)}
\eeq
has the property that the strict majorant property fails in $L^p([0,1]^d)$, since 
\[ \binom{p/2}{m_+^{(n_0)}}<0 \quad \text{yet} \quad  \binom{p/2}{m_-^{(n_0)}}>0.\]
Since we have obtained an infinite sequence of $n_0$ for which the corresponding values $m_+^{(n_0)}$ are arbitrarily large, this completes the proof of Theorem \ref{thm_affine_abun}. 

  \section{ Theorem \ref{thm_moment_curve}: the moment curve}
  For Theorem \ref{thm_moment_curve}, we let $\gamma(t) = (t,t^2,\ldots, t^d)$ parametrize the moment curve in $\R^d$. 
  We will take any $p >0$ and then choose $k$ depending on $p$ so that  the set $\Ga = \{\zerobf, \gamma(k), \gamma(k+1), \ldots, \gamma(k+d)\} \subset \Z^d$ violates the strict majorant property on $L^p([0,1]^d)$. We will again apply the analysis of \S \ref{sec_setup}, but since this time   we will apply the criterion (ii) to prove (\ref{pc}), we must do more explicit computations, which use the structure of the moment curve. 
  
  For the moment we suppose an integer $k\geq 1$ has been fixed. 
For each $0 \leq i \leq d$ define 
$n_i=\ga(k+i) \in \Z^d,$
regarded as a column vector. Accordingly, define the vector $v = (v_0,v_1,\ldots, v_d) \in \Z^{d+1}$ as before in (\ref{1v_generic}) and (\ref{affine_cond}).

\begin{lemma}\label{201212claim3_1} 
Fix an integer $k \geq 1$. Let $v = (v_0,v_1,\ldots, v_d) \in \Z^{d+1}$ be defined as above. If we define 
 \[ c = \frac{1}{d! (d-1)! \dots 1!} v ,\]
then $c \in \Z^{d+1}$  and 
\begin{equation}\label{201212e3_3}
c_0 + \dots + c_d
= 1,
\end{equation}
so   $|c| = |c_0| + \cdots + |c_d|$ is odd.
Also, $|c_i|\ge k^d/(d!)^2$ for every $0\le i\le d$. 
\end{lemma}
Assuming this lemma, we prove Theorem \ref{thm_moment_curve}.
 We  claim that for any $p >0$, $p \not\in 2\N$, there exists a choice of $k$ such that for $v$ as defined above, and hence for $c = c_+ - c_-$ (with $c_+,c_-$ defined as in (\ref{c_notation})), the crucial relation (\ref{pc}) holds. 
 Since $|c| = |c_+| + |c_-|$ is odd, then  $|c_-|$ and $|c_+|$   have opposite parity. Thus in order to confirm that (\ref{pc}) holds, it only remains to show that given any such $p$ we can take $k$ sufficiently large so that the vector $c$ constructed above leads to     $|c_-|, |c_+| > p/2$. This follows immediately from the last statement in Lemma \ref{201212claim3_1}. 
This verifies (\ref{pc}) for the set of points $\{n_0,n_1,\ldots,n_d\}= \{\gamma(k), \gamma(k+1), \ldots, \gamma(k+d)\}$. 
Finally, we choose $a$ so that (\ref{ac}) holds and $|a|$ is sufficiently small that the error $o(|a^{c_- + c_+}|)$ in Proposition \ref{prop_summary} can be dominated by the main term on the right-hand side. (There is a positive measure set of such $a$.) This proves Theorem \ref{thm_moment_curve}.
\subsection{Proof of Lemma \ref{201212claim3_1}}
We will first show that 
\beq\label{v_sum_moment}
v_0+v_1 + \cdots +v_d = d! (d-1)! \cdots 1!.
\eeq
Then  we will compute each coordinate $v_i$, and show that 
\beq\label{v_divis}
(d! (d-1)! \cdots 1!) | v_i \qquad \text{for each $0 \leq i \leq d$.}
\eeq
Consequently  $c$ is an integer vector as claimed, and from (\ref{v_sum_moment}) we immediately see that (\ref{201212e3_3}) holds. By Bezout's identity, (\ref{201212e3_3})  implies that $\gcd(c_0,\ldots,c_d)=1$ so we can incidentally conclude that 
$
\gcd(v_0, \dots, v_d)= d! (d-1)! \cdots 1!.
$
Thus with this definition of $v$ and $c$ we can apply all the analysis of \S \ref{sec_setup}. Moreover, by (\ref{201212e3_3}),  $|c| = |c_0| + |c_1| + \cdots + |c_d|$ must be odd, since for each $i$, $c_i$ and $|c_i|$  have the same parity. We will prove the last claim of the lemma by examining the expression we prove for each $c_i.$ 

We now prove (\ref{v_sum_moment}). Recall the expression for $(1,1,\ldots,1) \cdot (v_0,v_1,\ldots,v_d)$ in (\ref{affine_cond}).
By the standard Vandermonde determinant identity, we claim that $v_0 + v_1 + \cdots + v_d$ must equal
  \beq\label{claim_v_moment}
  \det \left( 
\begin{array}{cccc}
1 & 1 & \dots & 1 \\
k & k+1 & \dots & k+d \\
\vdots & & & \vdots \\
k^d & (k+1)^d & \dots & (k+d)^d 
\end{array}
\right) = d! (d-1)! \cdots 1! .
\eeq
Indeed, by the Vandermonde  identity    the  $(d+1)\times (d+1)$ determinant is 
\[
\prod_{0 \leq j' < j \leq d} ( (k+j) - (k-j')) = \prod_{j=1}^d \left( \prod_{j'=0}^{j-1} (j-j') \right) = \prod_{j=1}^d j!
\]
as claimed.

Now we derive an expression for each coordinate $v_i$ individually, using the definition for $v$ given in (\ref{1v_generic}), so that $v_i$ is obtained from the minor of the $(d+1)\times (d+1)$ matrix $(\widetilde{n_0},\widetilde{n_1},\ldots,\widetilde{n_d})$ that omits the top row and the $i$-th column, numbering the columns from $0$ to $d$. For each $0 \leq i \leq d$, we see that $(-1)^i v_i$ is precisely 
\begin{multline*}
 \det \left(
\begin{array}{cccccc}
\gamma(k) & \cdots & \gamma(k+i-1) & \gamma(k+i+1) & \dots & \gamma(k+d)
\end{array}
\right) \\
= k \cdots (k+i-1) (k+i+1) \dots (k+d)
\mathbf{D}_i
\end{multline*}
where 
\[ \mathbf{D}_i=
\det \left(
\begin{array}{cccccc}
1 & \dots & 1 & 1 & \cdots & 1 \\
k & \dots & k+i-1 & k+i+1 & \cdots & k+d \\
\vdots \\
k^{d-1} & \dots & (k+i-1)^{d-1} & (k+i+1)^{d-1} & \cdots & (k+d)^{d-1} \!
\end{array}
\right).
\]
(When $i=0$,  the product before $\Dbf_0$ is $(k+1)\cdots (k+d)$ and the first column in $\Dbf_0$ is $(1, k+1,\ldots, (k+1)^{d-1})$, and so on.) First fix some $1 \leq i \leq d$.
Applying the Vandermonde identity again, 
\[
\mathbf{D}_i=\prod_{1 \leq j' < j \leq d} (w_j - w_{j'})
\]
where $w_j := k + j - 1$ if $j \leq i$, and $w_j := k + j$ if $j \geq i+1$. Now for $1 \leq j \leq d$, 
\[
\prod_{1 \leq j' < j} (w_j - w_{j'}) = 
\begin{cases}
 (j-1)! &\quad \text{if $1 \leq j \leq i$} \\
 \frac{j!}{j-i} &\quad \text{if $i < j \leq d$}.
 \end{cases}
\]
Multiply over all $1 \leq j \leq d$, and simplify, obtaining
\[ v_i=(-1)^i \binom{k+i-1}{i} \binom{k+d}{d-i} d! (d-1)! \cdots 1!.\]
Consequently, $d! (d-1)! \dots 1!$ divides $v_i$  and $c_i$ is given explicitly as
\[
c_i = (-1)^i \frac{k \cdots (k+i-1)}{i!} \frac{(k+i+1) \cdots (k+d)}{(d-i)!},
\]
which implies the desired lower bound on $|c_i|$. 

For $i=0$, the computation is similar.
To compute $\mathbf{D}_0$ we define
 $w_j := k + j$ for all $j \geq 1$, and then for each $1 \leq j \leq d$, 
\[
\prod_{1 \leq j' < j} (w_j - w_{j'}) = (j-1)!.
\]
So arguing as before, we see
\[ v_0= \binom{k+d}{d} d! (d-1)! \cdots 1!, \qquad 
    c_0 = \binom{k+d}{d} = \frac{(k+1) \cdots (k+d)}{d!}
.\]
This completes the proof of (\ref{v_divis}) and hence the lemma is proved.

\section{Cases when majorant properties hold}\label{sec_remarks}

\subsection{Strict majorant property holds when \texorpdfstring{$p \in 2\N$}{p is even}}\label{sec_2k}
Here we simply observe that the strict majorant property holds in great generality  when $p \in 2\N$. (See Bennett and Bez \cite{MR3050801} for the case $\phi(n) = (n,n^2).$)
\begin{lemma}\label{lem:majorant}
For every $\phi:\N^k\to \Z^d$ and every positive integer $s$,
\begin{equation}\label{3-majorant}  \sup_{|a_n|\le A_n} \Big\| \sum_{n\in [1,N]^k} a_n e(\phi(n) \cdot \alpha ) \Big\|_{L^{2s}([0,1]^d)}	\le  \Big\| \sum_{n\in [1,N]^k} A_n e(  \phi(n) \cdot \alpha) \Big\|_{L^{2s}([0,1]^d)}.
\end{equation}
\end{lemma}

\begin{proof}
Expanding out, we have
\begin{equation}\label{3-majorantpf1}
\int_{[0,1]^d} \Big| \sum_{n\in [1,N]^k} a_n e(  \phi(n) \cdot \alpha)\Big|^{2s} d\alpha = \sideset{}{^*}\sum_{(x,y)} a_{x_1}\cdots a_{x_s} \overline{a_{y_1} \cdots a_{y_s}},
\end{equation}
where $\sum^*_{(x,y)}$ refers to summation only over those pairs $(x,y)\in ([1,N]^k)^s \times ([1,N]^k)^s$ of integral tuples that satisfy
\begin{equation}\label{3-majorantpf2}
\phi(x_1)+\cdots+\phi(x_s) = \phi(y_1)+\cdots+\phi(y_s).
\end{equation}
 Since $|a_n|\le A_n$, the non--negative number \eqref{3-majorantpf1} is bounded above by
\[ \sideset{}{^*}\sum_{(x,y)} A_{x_1}\cdots A_{x_s} A_{y_1} \cdots A_{y_s} = \int_{[0,1]^d} \Big| \sum_{n\in [1,N]^k} A_n e(  \phi(n) \cdot \alpha)\Big|^{2s} d\al.\]
\end{proof}

\subsection{Strict majorant property holds for affinely independent sets} 

\begin{prop}\label{prop_characterize}
Let $\Ga \subset \Z^d$ be non-empty and affinely independent. Then $\Ga$ satisfies the strict majorant property for every $p > 0$.
\end{prop}

\begin{proof}
   Without loss of generality we may assume that $\Ga$ has affine dimension $d$ (since otherwise we could complete $\Ga$ to an affine dimension $d$ set). 
 Then the cardinality of $\Ga$ is $d+1$, and by translation invariance  of the proposed inequality (\ref{smp}) with respect to $\Ga$, we may assume $\Ga = \{\zerobf, n_1, \dots, n_d\}$ where $n_1, \dots, n_d \in \Z^d$ are linearly independent. 
Thus one can form an invertible $d\times d$ matrix $A$ with integer entries 
so that $n_j = A e_j$ for every $1 \leq j \leq d$; here $e_j$ is the $j$-th coordinate vector in $\Z^d$. Thus we may apply Lemma~\ref{lemma_change_var} and assume that $\Ga = \{\zerobf, e_1, \dots, e_d\}$. It suffices to show that whenever $p > 0$,  $a_0, a_1, \dots, a_d \in \mathbb{C}$ and $A_0,A_1,\dots,A_d \in \R$ with $|a_j| \leq A_j$ for all $j$, we have
\[
\int_{[0,1]^d} \! |a_0+a_1 e(x_1) + \dots + a_d e(x_d)|^p dx \leq \! \int_{[0,1]^d} \! |A_0+A_1 e(x_1)+ \dots + A_d e(x_d)|^p dx.
\]
Let 
\[
I(a_0,\dots,a_d) := \int_{[0,1]^d} |a_0+a_1 e(x_1) + \dots +a_d e(x_d)|^p dx.
\]
It is easy to see that $I$ is a symmetric in $a_0,\dots,a_d$: if $a_0',\dots,a_d'$ is a permutation of $a_0,\dots,a_d$, then 
$
I(a_0,\dots,a_d) = I(a_0',\dots,a_d').
$
(For instance, if $d = 2$, then $I(a_0,a_1,a_2) = \int_{[0,1]^2} |a_0 e(-x_1) + a_1 + a_2 e(x_2-x_1)|^p dx = I(a_1,a_0,a_2)$.) Furthermore, 
$
I(a_0,\dots,a_d) = I(|a_0|,\dots,|a_d|)
$
for every $a_0,\dots,a_d \in \mathbb{C}$, by periodicity of the integrand defining $I$. (For example, if $a_0 e(\al_0) \in \R$, we first multiply the integrand by $|e(\al_0)|^p$, and then set $x_1 \mapsto x_1+\al_1$ for $\al_1 \in [0,1]$ such that $a_1 e(\al_0+\al_1) \in \R$, and so on.)
Thus we can think of $I(a_0,\dots,a_d)$ as a function defined on $[0,\infty)^{d+1}$. Taking both the above properties into account, it suffices to show that this function is non-decreasing in $a_d \in [0,\infty)$ once $a_0, \dots, a_{d-1}$ are fixed. This follows from applying the lemma below   with $a = a_0 + a_1 e(x_1) + \dots + a_{d-1} e(x_{d-1})$, $b = a_d$ and $B = A_d$.
\end{proof}

\begin{lemma}
Suppose $p > 0$, $a \in \mathbb{C}$,  and $0 \leq b \leq B$. Then 
\[
\int_0^1 |a + b e(t)|^p dt \leq \int_0^1 |a + B e(t)|^p dt.
\]
\end{lemma}

\begin{proof}
The assertion is clear when $a = 0$, and if $a \ne 0$ we can factor out $|a|^p$ from both sides, so without loss of generality we may assume $a = 1$. Then we need to show that the function
$
G(r) = \int_0^1 |1 + re(t)|^p dt
$
is a non-decreasing function of $r \in [0,\infty)$ for all $p>0$.
We compute that
\[G'(r) = p \int_0^1 (1+r^2+2r\cos(2\pi t))^{p/2-1}(r+\cos(2\pi t))\, dt.\]
For $r\ge 1$ it is evident that $G'(r)\ge 0$. 
For $r\in (0,1)$, $G(r)$ is represented by its Taylor series at $r=0$ as in \eqref{Taylor}, which has only non-negative coefficients. Invoking continuity of $G$ at $r=1$ finishes the proof that $G$ is non-decreasing.
\end{proof}

\subsection{A (weaker) majorant property holds for the moment curve}

Let $\ga(t) = (t,t^2, \ldots, t^d)$ parametrize the moment curve in $\R^d$, for $d \geq 2$.
Following the argument of
\cite[Thm. 1.2]{MR3050801} for $d=2$, we prove (\ref{weak_constant}).
The key step is to show that for any sequence of real coefficients $b=\{b_n\},$ for any integer $1 \leq r \leq d,$
\beq\label{2d2}
 \| b\|_{\ell^2(\Z)} \leq \Big\| \sum_{n \in \Z} b_ne(\ga(n) \cdot x) \Big\|_{L^{2r}([0,1]^d)}
     \leq c_r \| b\|_{\ell^2(\Z)}.
\eeq
    In fact we can take $c_r=(r!)^{1/2r}.$
Once (\ref{2d2}) is known, applying it for the choices $r=d$ and $r=1,$ along with two applications of H\"older's inequality, and the assumption on $|a_n| \leq A_n$, shows that   for any $2 \leq p \leq 2d,$
\begin{multline*}
\Big\| \sum_{n \in \Z} a_ne(\ga(n) \cdot x) \Big\|_{L^{p}([0,1]^d)}
    \leq \Big\| \sum_{n \in \Z} a_ne(\ga(n) \cdot x) \Big\|_{L^{2d}([0,1]^d)}
     \leq c_d \| a\|_{\ell^2(\Z)}\\
     \leq c_d\| A\|_{\ell^2(\Z)} \leq c_d \Big\| \sum_{n \in \Z} A_ne(\ga(n) \cdot x) \Big\|_{L^{2}([0,1]^d)}
    \leq c_d \Big\| \sum_{n \in \Z} A_ne(\ga(n) \cdot x) \Big\|_{L^{p}([0,1]^d)}.
\end{multline*}

To prove (\ref{2d2}), fix an integer $1 \leq r \leq d$ and expand the $2r$-th power of the central expression. It is equal to 
\[ \sideset{}{^*}\sum_{\substack{(n_1, \ldots, n_r) \in \Z^r\\(m_1,\ldots,m_r) \in \Z^r}} b_{n_1}\cdots b_{n_r} \overline{b_{m_1}} \cdots \overline{b_{m_r}},\]
in which the restricted summation is over tuples of $n_i, m_i \in \Z$ that satisfy the Vinogradov system of $d$ simultaneous equations in $2r$ variables,
\[ n_1^j + \cdots + n_r^j = m_1^j + \cdots + m_r^j, \qquad 1 \leq j \leq d.\]
If  $r \leq d$, the only integral solutions to this are diagonal, that is,  the tuple $(n_1,\ldots,n_r)$ is a permutation of $(m_1,\ldots,m_r)$ (see e.g. \cite[Lemma 2.1]{GGPRY}).
Thus upon letting $N(n_1,\ldots,n_r)$ denote the number of tuples $(m_1,\ldots, m_r)$ that are permutations of $(n_1,\ldots, n_r)$,
\[\Big\| \sum_{n \in \Z} b_ne(\ga(n) \cdot x) \Big\|_{L^{2r}([0,1]^d)}^{2r}
     = \sum_{(n_1,\ldots, n_r) \in \Z^r} N(n_1,\ldots, n_r) |b_{n_1}|^2 \cdots |b_{n_r}|^2.
     \]
     This is bounded above by  $r! \| b\|_{\ell^2(\Z)}^{2r}$ and below by $\| b\|_{\ell^2(\Z)}^{2r}$, verifying (\ref{2d2}).

\subsection*{Acknowledgements}
We thank AIM for funding our SQuaRE workshop. Gressman was partially supported by NSF  DMS-1764143; Guo   by 
NSF DMS-1800274; 
Pierce  by NSF CAREER DMS-1652173, a Sloan Research Fellowship, and an AMS Joan and Joseph Birman Fellowship; Yung   by a Future Fellowship FT200100399 from the Australian Research Council.

\newcommand{\etalchar}[1]{$^{#1}$}

\end{document}

%% file: format.tex
\newtheorem{thm}{Theorem}
\newtheorem*{thm*}{Theorem}
\newtheorem{prop}[thm]{Proposition}
\newtheorem{lemma}[thm]{Lemma}

\theoremstyle{remark}

\newcommand{\Q}{\mathbb{Q}}
\newcommand{\N}{\mathbb{N}}
\newcommand{\R}{\mathbb{R}}

\newcommand{\Z}{\mathbb{Z}}

\newcommand{\GL}{\mathrm{GL}}

\newcommand{\ep}{\varepsilon}

\newcommand{\maps}{\rightarrow}

\newcommand{\al}{\alpha}
\newcommand{\be}{\beta}

\newcommand{\ga}{\gamma}

\newcommand{\Ga}{\Gamma}

\newcommand{\Dbf}{\mathbf{D}}

\newcommand{\zerobf}{\boldsymbol0}

\newcommand{\beq}{\begin{equation}}
\newcommand{\eeq}{\end{equation}}

\makeatletter
\def\@tocline#1#2#3#4#5#6#7{\relax
  \ifnum #1>\c@tocdepth 
  \else
    \par \addpenalty\@secpenalty\addvspace{#2}%
    \begingroup \hyphenpenalty\@M
    \@ifempty{#4}{%
      \@tempdima\csname r@tocindent\number#1\endcsname\relax
    }{%
      \@tempdima#4\relax
    }%
    \parindent\z@ \leftskip#3\relax \advance\leftskip\@tempdima\relax
    \rightskip\@pnumwidth plus4em \parfillskip-\@pnumwidth
    #5\leavevmode\hskip-\@tempdima
      \ifcase #1
       \or\or \hskip 1em \or \hskip 2em \else \hskip 3em \fi%
      #6\nobreak\relax
    \hfill\hbox to\@pnumwidth{\@tocpagenum{#7}}\par
    \nobreak
    \endgroup
  \fi}
\makeatother

%% file: majorant_arxiv_v2.bbl
\begin{thebibliography}{GGP{\etalchar{+}}21}

\bibitem[Art91]{MR1129886}
M.~Artin.
\newblock {\em Algebra}.
\newblock Prentice Hall, Inc., Englewood Cliffs, NJ, 1991.

\bibitem[Bac73]{MR320636}
G.~F. Bachelis.
\newblock On the upper and lower majorant properties in {$L^{p}(G)$}.
\newblock {\em Quart. J. Math. Oxford Ser. (2)}, 24:119--128, 1973.

\bibitem[BB12]{MR3050801}
J.~Bennett and N.~Bez.
\newblock A majorant problem for the periodic {S}chr\"{o}dinger group.
\newblock In {\em Harmonic analysis and nonlinear partial differential
  equations}, RIMS K\^{o}ky\^{u}roku Bessatsu, B33, pages 1--10. (RIMS), Kyoto,
  2012.

\bibitem[BBC09]{BBC09}
J.~Bennett, N.~Bez, and A.~Carbery.
\newblock Heat-flow monotonicity related to the {H}ausdorff-{Y}oung inequality.
\newblock {\em Bull. Lond. Math. Soc.}, 41(6):971--979, 2009.

\bibitem[Ber12]{Ber12}
P.~Bernays.
\newblock {\"Uber die Darstellung von positiven, ganzen Zahlen durch die
  primitiven bin\"aaren quadratischen Formen einer nichtquadratischen
  Diskriminante}.
\newblock {\em Dissertation, G\"ottingen}, 1912.

\bibitem[BG06]{BloGra06}
Valentin Blomer and Andrew Granville.
\newblock Estimates for representation numbers of quadratic forms.
\newblock {\em Duke Math. J.}, 135(2):261--302, 2006.

\bibitem[Blo05]{Blo05}
Valentin Blomer.
\newblock Binary quadratic forms with large discriminants and sums of two
  squareful numbers. {II}.
\newblock {\em J. London Math. Soc. (2)}, 71(1):69--84, 2005.

\bibitem[Boa63]{MR149175}
R.~P. Boas, Jr.
\newblock Majorant problems for trigonometric series.
\newblock {\em J. Analyse Math.}, 10:253--271, 1962/63.

\bibitem[Bou89]{MR1029904}
J.~Bourgain.
\newblock On {$\Lambda(p)$}-subsets of squares.
\newblock {\em Israel J. Math.}, 67(3):291--311, 1989.

\bibitem[DL21]{DemLan21x}
C.~Demeter and B.~Langowski.
\newblock Restriction of exponential sums to hypersurfaces.
\newblock {\em arXiv:2104.11367v2}, 2021.

\bibitem[Gre05]{MR2180408}
B.~Green.
\newblock Roth's theorem in the primes.
\newblock {\em Ann. of Math. (2)}, 161(3):1609--1636, 2005.

\bibitem[GR04]{MR2103913}
B.~Green and I.~Z. Ruzsa.
\newblock On the {H}ardy-{L}ittlewood majorant problem.
\newblock {\em Math. Proc. Cambridge Philos. Soc.}, 137(3):511--517, 2004.

\bibitem[GGP{\etalchar{+}}21]{GGPRY}
P.~T. Gressman, S.~Guo, L.~B. Pierce, J.~Roos, and P.-L. Yung.
\newblock Reversing a philosophy: from counting to square functions and
  decoupling.
\newblock {\em J. Geom. Analysis}, (in press), 2021.

\bibitem[HL35]{HarLit35}
G.~H. Hardy and J.~E. Littlewood.
\newblock Notes on the theory of series {XIX}: a problem concerning majorants
  of {F}ourier series.
\newblock {\em Quart. J. Math.}, 6:304--315, 1935.

\bibitem[KMT16]{MR3494246}
Ben Krause, Mariusz Mirek, and Bartosz Trojan.
\newblock On the {H}ardy-{L}ittlewood majorant problem for arithmetic sets.
\newblock {\em J. Funct. Anal.}, 271(1):164--181, 2016.


\bibitem[Lan09]{Lan09}
E.~Landau.
\newblock {\em Handbuch der Lehre der Primzahlverteilung}.
\newblock Teubner, Leipzig, 1909.

\bibitem[Lan02]{MR1878556}
S.~Lang.
\newblock {\em Algebra}, volume 211 of {\em Graduate Texts in Mathematics}.
\newblock Springer-Verlag, New York, third edition, 2002.

\bibitem[Moc96]{Mockenhaupt}
G.~Mockenhaupt.
\newblock Bounds in {L}ebesgue spaces of oscillatory integral operators.
\newblock {\em Habilitationsschrift, Univ.-GHS Siegen, Siegen}, 1996.

\bibitem[MS09]{MR2488338}
G.~Mockenhaupt and W.~Schlag.
\newblock On the {H}ardy-{L}ittlewood majorant problem for random sets.
\newblock {\em J. Funct. Anal.}, 256(4):1189--1237, 2009.

\bibitem[Zyg59]{MR0107776}
A.~Zygmund.
\newblock {\em Trigonometric series. 2nd ed. {V}ols. {I}, {II}}.
\newblock Cambridge University Press, New York, 1959.

\end{thebibliography}
